\documentclass[11pt]{article}
\usepackage{amsmath,amssymb,amsbsy,amsfonts,amsthm,latexsym,
            amsopn,amstext,amsxtra,euscript,amscd,amsthm}
\usepackage{color}
\newtheorem{lem}{Lemma}
\newtheorem{lemma}[lem]{Lemma}

\newtheorem{prop}{Proposition}
\newtheorem{proposition}[prop]{Proposition}

\newtheorem{thm}{Theorem}
\newtheorem{theorem}[thm]{Theorem}

\newtheorem{exam}{Example}
\newtheorem{example}[exam]{Example}

\newtheorem{prob}{Problem}
\newtheorem{problem}[prob]{Problem}

\def\\{\cr}
\def\({\left(}
\def\){\right)}
\def\[{\left[}
\def\]{\right]}
\def\<{\langle}
\def\>{\rangle}

\def\lcm{\mathrm{ lcm}}
\def\deg{\mathrm{deg}}

\def\N{{\mathbb N}}
\def\Z{{\mathbb Z}}

\def\Q{{\mathbb Q}}

\def\C{{\mathbb C}}

\begin{document}

\title{Coprime partitions and Jordan totient functions}
\author{\textbf{Daniela Bubboloni} \\
%EndAName
{\small {Dipartimento di Matematica e Informatica U.Dini} }\\
\vspace{-6mm}\\
{\small {Universit\`{a} degli Studi di Firenze} }\\
\vspace{-6mm}\\
{\small {viale Morgagni 67/a, 50134 Firenze, Italy}}\\
\vspace{-6mm}\\
{\small {e-mail: daniela.bubboloni@unifi.it}}\\
\vspace{-6mm}\\
{\small https://orcid.org/0000-0002-1639-9525} \and \textbf{Florian Luca}\footnote{Also affiliated with the Research Group in Algebraic Structures and Applications, King Abdulaziz University, Jeddah, Saudi Arabia and Centro de Ciencias Matem\'aticas UNAM, Morelia, Mexico}
\\
%EndAName
{\small {School of Mathematics} }\\
\vspace{-6mm}\\
{\small {University of the Witwatersrand} }\\
\vspace{-6mm}\\
{\small {1 Jan Smuts, Braamfontein 2000, Johannesburg, South Africa}}\\
\vspace{-6mm}\\
{\small {e-mail: Florian.Luca@wits.ac.za }}\\
\vspace{-6mm}\\
{\small https://orcid.org/0000-0003-1321-4422}\\
\vspace{-6mm}\\}

\date{\today}

\pagenumbering{arabic}

\maketitle
\begin{abstract} We show that while  the number of  coprime compositions of a positive integer $n$ into $k$ parts can be expressed as a $\Q$-linear combinations of the Jordan totient functions,  this is never possible for  the coprime partitions of $n$  into $k$ parts.  We also show that the number $p_k'(n)$ of coprime partitions of $n$ into $k$ parts can be expressed as a $\C$-linear combinations of the Jordan totient functions, for $n$ sufficiently large, if and only if $k\in \{2,3\}$ and in a unique way. Finally we introduce some generalizations of the Jordan totient functions and we show that $p_k'(n)$ can be always expressed as a $\C$-linear combinations of them.\end{abstract}
 \noindent

\section{Introduction}\label{intro}
The study of partitions with a fixed number $k$ of parts satisfying some coprimality condition  \cite{BLS} has revealed to be very fruitful for analysing the normal covering number $\gamma(S_n)$ of the symmetric group $S_n$
 \cite{BPS}, that is, the smallest number of conjugacy classes of proper subgroups needed to cover $S_n$. If $\sigma\in S_n$ and $k$ is the number of orbits  of $\langle\sigma\rangle$ on $\{1,\dots,n\}$ then the unordered list $\mathfrak{p}(\sigma)=[x_1,\ldots,x_k]$ of the sizes $x_i$  of those orbits is a partition of $n$ into $k$ parts called the type of $\sigma.$ Now, by a basic result of group theory, two permutations are conjugate if and only if they have the same type.  Thus, the conjugates of some subgroups $H_1 , \ldots , H_s $ cover $S_n$ if and only if for every partition  $\mathfrak{p}$ of $n$ there exists $H_i$ containing at least a permutation of type  $\mathfrak{p}$.  We emphasize that the problem of determining the normal covering number of a finite group arises from Galois theory and is linked to the investigation of integer polynomials having a root modulo $p$, for every prime number $p$ (see~\cite[Section 1]{BBH}, ~\cite{BS} and ~\cite{RS}).

 Fortunately, in order to efficiently  bound $\gamma(S_n)$, it is not necessary to deal with partitions into $k$ parts for every possible $k\leq n$ and the focus is on $k=2,3,4$ (see \cite[Sections 5-6]{JCTA} and  \cite[Remark 1.2\,(c) and Sections 6-7]{BPS-19}). Recently, using knowledge about partitions into three parts  Bubboloni, Praeger and Spiga \cite[Theorem 1.1]{BPS-19} have shown that, for $n$ even, $\gamma(S_n)\ge
\frac{n}{2}\left(1-\sqrt{1- 4/\pi^2}\right)-\frac{\sqrt{17}}{2}n^{3/4}.$
Similar results about $S_n$ for $n$ odd are not known and the research could greatly benefit from knowing more about partitions into four parts, especially those satisfying suitable coprimality conditions. A point of force in this direction is the fact that the primitive subgroups of $S_n$ containing a permutation splitting into four cycles have been recently determined \cite{CPS}.
To start with, one should find an exact formula for the number $p_4'(n)$ of coprime partitions of $n$ into four parts. This initial and somewhat narrow  motivation inspired the present paper. 

Looking to the case $k=4$, we immediately realized that many considerations could be indeed carried on for every $k\geq 2$, shedding light on the number $p'_k(n)$ of coprime partitions of $n$ into $k$ parts. 
The idea relies on one hand, on representing those expressions as linear combinations of classic number theoretic functions and, on the other hand, on having a method which leads to an effective computation of $p_k(n)$ and $p'_k(n)$. This has appeared in the past in a number of papers concerning $p_k(n)$ but we did not see it for $p_k'(n)$.  In fact a formula for  $p'_k(n)$ seems to be of recent interest in the scientific community (see \cite[Question 2]{Man}).

Let $J_i$ denote the Jordan totient function of degree $i\geq 0$.
In \cite{bach}, it is proved that  $p'_3(n)=\frac{J_2(n)}{12}$ holds for $ n\geq 4$. It is also clear that $p'_2(n)=\frac{J_1(n)}{2}$ holds for $ n\geq 3.$ So, one can ask if similar results could hold for every $k.$ We show that those two situations are pure miracles, because $p_k'(n)$  is in fact a $\mathbb{C}$-linear combination of the Jordan totient functions for $n$ sufficiently large just in those two cases (Theorem \ref{comb-genera}). The feeling is that the class of  the Jordan totient functions is too restrictive and some generalizations of them are needed. We consider then three generalizations which are finely linked together:  the Jordan root totient function,  the Jordan modulo totient function and  the Jordan-Dirichlet totient functions (Section \ref{Jordan}). The first two generalizations seem not to be present in the literature. The third ones appeared in \cite{cyc} in order to investigate the values of the cyclotomic polynomial at the roots of unity and admit easy and manageable formulas. We show that $p_k'$ is a $\C$-linear combination of the Jordan root totient functions (Theorem \ref{comb-genera}). Relying on the partial fraction decomposition of the generating function of $p_k(n)$ and classical results about linearly recurrent sequences, we explicitly find the coefficients of such $\C$-linear combination and show how to deduce the expression of the Jordan root totient functions involved. To that last purpose the idea is to split a Jordan root totient function in a $\C$-linear combination of Jordan modulo totient functions, which in turn can be determined by suitable Jordan-Dirichlet totient functions, choosing some particular Dirichlet characters.  Our concrete approach is proposed in detail for $k\in\{2,3,4\}$.

We close noticing that the use  of generalizations of Jordan totient functions is present in the very recent research. For instance in \cite{M}, Moree et al. introduce the Jordan totient quotients of weight $w$ in order to study the average of the normalized derivative of cyclotomic polynomials.

\section{Basic facts}
\subsection{Notation}
We denote by $\mathbb{N}$ the set of positive integers and by $\mathbb{N}_0$ the set of non-negative integers. Let $n\in\N.$  We denote by $\Omega(n)$  the number of prime factors of $n$ counted with multiplicity and by $\omega(n)$  the number of distinct prime factors of $n$, where $\Omega(1)=\omega(1)=0$.
Moreover we define $\delta(n)=\lcm\{m\in  \mathbb{N}: m\le n \}$. As usual, $\phi$ denotes the Euler's totient function and $\mu$ the M\"obius function. For $k\in \N_0$, set $[k]=\{n\in \N : n\leq k\}$ and $[k]_0=\{n\in \N_0 : n\leq k\}$. In particular, $[0]=\varnothing$ while $[0]_0=\{0\}.$

Let $f:\mathbb N_0 \rightarrow {\mathbb C}$. Then $f$ is called an integer periodic function if 
$$M(f):=\{m'\in \N: \forall n,k\in \N_0,\, f(n+km')=f(n)\}\neq \varnothing.$$ 
The number $m:=\min M(f)$ is the period of $f$ and $M(f)=\{km: k\in \N\}$.

The function $f$ is called a  quasi--polynomial of degree $d\in \mathbb{N}_0$ if, for every $j\in [d]_0$, there exists an integer periodic function $f_j$ with period $m_j\in \mathbb N$ and $f_d$ not identically zero such that
  \begin{equation*}\label{q-p}
  f(n)=\sum_{j=0}^df_j(n)n^j\quad  \hbox{for all}\quad  n\in  \N_0.
  \end{equation*}
The minimum positive integer in $\bigcap_{j=0}^d M(f_j)$ equals
   $\lcm\{m_j:j\in [d]_0\}$ and is called the quasi--period of $f$.
Note that  the quasi-polynomials form a vectorial space over $\C$ which includes the integer periodic functions as well as the polynomials.

Given a sequence $(a_n)_{n\geq k}$ of complex numbers for some $k\in \N_0$, its generating function is the formal power series 
$$
\sum_{n\geq k}a_n z^n.
$$ 
With one exception (Proposition \ref{recurr}), in all instances we are treating in this paper, $a_n$ has polynomial growth. That is, $|a_n|=O(n^s)$ holds for all $n\ge k$ with some  $s\in \N$. 
 In the exceptional case $a_n$ has exponential growth; that is $|a_n|=\exp(O(n))$. Thus, the power series above has the radius of convergence at least 
$1$ in all cases except  the exceptional case for which the radius of convergence is positive. So, we think of it as an analytic function in some open disk.
 
 For $n\in {\mathbb N}$ we denote the group of $n$-roots of unity $U_n=\{z\in  {\mathbb C}: z^n=1 \}. $ It is well known that $U_n$ is cyclic with $\phi(n)$ generators called primitive  $n$-roots of unity. Among the primitive  $n$-roots of unity $e^{\frac{2\pi i}{n}}$ is called the principal  $n$-root of unity.
Every  $\omega\in U:=\bigcup_{n\in \N} U_n$ is called a root of unity. 
  If $P(X)\in{\mathbb C}[X]$, we denote its degree by $\mathrm{deg}(P).$

\subsection{ The Jordan totient functions and their generalizations}\label{sec:Jordan}

Throughout this section, let $k$ be a non-negative integer.
We first recall the basic properties of the {\it Jordan totient function} $J_k:\mathbb{N}\rightarrow \mathbb{N}_0$ of degree $k$. For every $n\in \mathbb{N}$, by definition, we have
\begin{equation}
\label{Jordan}
J_k(n):=\sum_{d\mid n}d^k \mu(n/d).
\end{equation}
Note that $J_k$ is a Dirichlet convolution of multiplicative functions, and thus it is a multiplicative function. Moreover,
\begin{equation}\label{Jord0}
J_0(n)=\sum_{d\mid n} \mu(n/d)=\left\{\begin{matrix} 1 & {\text{\rm if}} & n=1\\ 0 & {\text{\rm if}} & n>1,\end{matrix}\right.
\end{equation}
 is the neutral element with respect to the Dirichlet $*$-product of arithmetic functions.

The values of $J_k(n)$  for $k\geq 1$ can be easily computed in terms of the prime divisors of $n$ by the formula
\begin{equation*}\label{Jord-comp}
J_k(n)=n^k\prod_{p\mid n}\left(1-\frac{1}{p^k}\right),
\end{equation*}
which makes clear that
\begin{equation*}\label{Jord1}
J_1(n)=\sum_{d\mid n} d\mu(n/d)=\phi(n).
\end{equation*}

For the scope of our paper it is fundamental to consider some variations of the Jordan totient functions.

We define, for a root of unity $\omega$, the $\omega$-Jordan totient function of degree $k$ denoted $J_{k,\omega}:\mathbb{N}\rightarrow \mathbb{C}$
which associates to $n\in \mathbb{N}$ the complex number
\begin{equation}\label{Jord-root}
J_{(k,\omega)}(n):=\sum_{d\mid n}\omega^d d^k\mu(n/d).
\end{equation}
We call those functions the {\it Jordan root totient functions}\footnote{The definition \eqref{Jord-root} can be given for a generic $\omega \in \mathbb{C}$, but that has no interest for the scope of the present paper.}.
Note these are generalisations of the Jordan totient functions because  $J_{(k,1)}=J_{k}$. However, those functions are not multiplicative when $\omega\neq1.$

We define next, for every $m\in \N$ and $j\in [m-1]_0$, the {\it Jordan modulo totient functions} of degree $k$ denoted $J_{k}^{j, m}:\mathbb{N}\rightarrow \mathbb{C}$ which associates to $n\in \mathbb{N}$ the integer
\begin{equation*}\label{Jord-mod}
J_{k}^{j, m}(n):=\sum_{\substack{d\mid n\\ d\equiv j\pmod m}} d^k\mu(n/d).
\end{equation*}

 Note that those functions cannot be interpreted as convolutions of multiplicative functions because the sum is not extended to all the divisors of $n$. In particular, they are not multiplicative in general.  Since $J_{k}^{0, 1}=J_k$ the  Jordan modulo totient functions are generalizations of the  Jordan totient functions as well.

 It is immediately observed that  the Jordan root totient functions are $\C$-linear combinations of the Jordan modulo totient functions. More precisely, consider $J_{(k,\omega)}$ for some $\omega\in U $ and some $k\in \N_0$.
 Let  $m$ be the minimum positive integer such that $\omega\in U_m.$ Then, for every $n\geq 1$, we have
 \begin{equation}\label{root-modulo}
 J_{(k,\omega)}(n)=\sum_{j=0}^{m-1}\omega^j\sum_{\substack{d\mid n\\ d\equiv j\pmod m}} d^k\mu(n/d)=\sum_{j=0}^{m-1}\omega^jJ_{k}^{j, m}(n).
 \end{equation}
 Thus, $J_{(k,\omega)}=\sum_{j=0}^{m-1}\omega^jJ_{k}^{j, m}.$

We finally recall a definition from \cite{cyc}.
For a Dirichlet character $\chi$, the function  $J_k(\chi;\cdot):\mathbb{N}\rightarrow \mathbb{C}$ is defined by associating to every $n\in \mathbb{N}$ the complex number
 \begin{equation*}\label{Jord-Dir}
 J_k(\chi;n) :=\sum_{ d\mid n}\chi(d)d^k\mu(n/d).
\end{equation*}
We call those functions the {\it Jordan-Dirichlet totient functions}.
Since $\chi$ is totally multiplicative, the function $ J_k(\chi;\cdot)$ is a Dirichlet convolution of multiplicative functions, and thus it is a multiplicative function. Note that if ${\bf 1}$ is the unique Dirichlet character modulo $1$ (called the  trivial character), that is the  function ${\bf 1}(x)=1$ for every $x\in \Z$, we have that $ J_k({\bf 1};\cdot)= J_k.$ % Recall that ${\bf 1}$ is the only character being non-zero at $0.$
Thus, the functions $ J_k(\chi;\cdot)$ are generalizations of the Jordan totient function $J_k.$
The values  $ J_k(\chi;n)$ can be explicitly computed when $\chi$ is assigned (see, for example, \cite[Lemma 6]{cyc}). Moreover, the Jordan-Dirichlet totient functions are  $\C$-linear combinations of the Jordan modulo totient functions.
\begin{lemma}
\label{char}
Let  $k$ be a non-negative integer and $\chi$ be a Dirichlet character modulo $m$ for $m$ a positive integer. Write $n\in \N$ as  $\displaystyle{n=\prod_{\substack{p^{c_p}\| n\\ c_p\geq 1,\, p\  \mathrm{prime}}} p^{c_p}}$. Then

\begin{itemize}
\item[$(i)$]
 \begin{equation*}\label{Jord-Dir-formula1}
  J_k(\chi;n) =n^k\prod_{\substack{p| n\\ p\  \mathrm{prime}}}\chi(p)^{c_p-1}\left(\chi(p)-\frac{1}{p^k}\right).
\end{equation*}
\item[$(ii)$] If $(n,m)=1$, then
 \begin{equation}\label{Jord-Dir-formula}
  J_k(\chi;n) =n^k\chi(n)\prod_{\substack{p| n\\ p\  \mathrm{prime}}}\left(1-\frac{1}{\chi(p)p^k}\right).
\end{equation}
\item[$(iii)$]
\begin{equation}\label{Jord-Dir-modulo}
  J_k(\chi;\cdot) =\sum_{j=0}^{m}\chi(j)J_{k}^{j, m}.
\end{equation}
\end{itemize}
\end{lemma}
\begin{proof} $(i)$ Using that $\chi$ is totally multiplicative, we have
\begin{eqnarray*}\label{prime}
  J_k(\chi;p^{c_p}) &=&\sum_{ d\mid p^{c_p}}\chi(d)d^k\mu(p^{c_p}/d)=-\chi(p^{c_p-1})p^{k(c_p-1)}+\chi(p^{c_p})p^{kc_p}\\
  &=& -\chi(p)^{c_p-1}p^{k(c_p-1)}+\chi(p)^{c_p}p^{kc_p}=\chi(p)^{c_p-1}p^{kc_p}\left(\chi(p)-\frac{1}{p^k} \right).
 \end{eqnarray*}
 Hence, by the multiplicativity of $ J_k(\chi;\cdot)$, we obtain
 \begin{eqnarray*}\label{general}
  J_k(\chi;n) &=&\prod_{\substack{p| n\\ p\  \mathrm{prime}}} J_k(\chi;p^{c_p})=\prod_{\substack{p| n\\ p\  \mathrm{prime}}}\chi(p)^{c_p-1}p^{kc_p}\left(\chi(p)-\frac{1}{p^k} \right)\\
  &=&n^k\prod_{\substack{p| n\\ p\  \mathrm{prime}}}\chi(p)^{c_p-1}\left(\chi(p)-\frac{1}{p^k}\right).
 \end{eqnarray*}

  $(ii)$ Since $(n,m)=1$ we have that,  for every prime $p$ dividing $n$, $\chi(p)\neq 0$ holds. Thus, the result
   follows immediately by $(i)$ using again that $\chi$ is totally multiplicative.

  $(iii)$ Since the Dirichlet characters modulo $m$ are periodic of period $m$, we have
  \begin{eqnarray*}
   J_k(\chi;n)& =&\sum_{d\mid n}\chi(d)d^k\mu(n/d)=\sum_{j=1}^m\sum_{\substack{d\mid n\\ d\equiv j\pmod m}}\chi(j) d^k\mu(n/d)\\
   &=&\sum_{j=1}^{m}\chi(j)J_{k}^{j, m}(n).
  \end{eqnarray*}
\end{proof}
We now briefly discuss how it possible to express the Jordan modulo totient functions by the Jordan-Dirichlet totient functions.

Recall that there are exactly $\phi(m)$ Dirichlet characters modulo $m$ so that, once $m$ is fixed, the equalities in \eqref{Jord-Dir-modulo} give $\phi(m)$ independent linear equations in the $m$ variables $ J_{k}^{j, m}$ for $j\in[m-1]_0$ with vanishing coefficient for the $j$ such that $\gcd(j,m)>1$. From those  equations  one can easily find the expression for $ J_{k}^{j, m}$, with $\gcd(j,m)=1$, in terms of the $J_k(\chi;\cdot).$  In fact, by the orthogonality relations for characters, we have
\begin{equation}
\label{eq:Jkjm}
J_k^{j,m}=\frac{1}{\phi(m)} \sum_{\chi} {\overline{\chi}}(j)J_k(\chi;\cdot)\qquad \hbox{for}\  (j,m)=1.
\end{equation}

The computation of  $J_k^{j,m}$ when $s:=\gcd(j,m)>1$ reduces to that of $J_k^{\,j/s,\,m/s}$ which, since $j/s$ and $m/s$ are coprime, is carried out through formula \eqref{eq:Jkjm}. More precisely we have
\begin{equation*}
\label{eq:Jkjm2}
J_k^{j,m}(n)=\left\{\begin{matrix} 0 & {\text{\rm if}} & s\nmid n\\ 
s^kJ_k^{\,j/s,\,m/s}(n/s) & {\text{\rm if}} & s\mid n.\end{matrix}\right.
\end{equation*}
Indeed, let  $j'=j/s$ and $m'=m/s$. 
 If $d$ is a positive integer such that $d\mid n$ and $d\equiv j\pmod m$, then we 
 have 
 \begin{equation}\label{d}
 d=sj'+ksm'=s(j'+km')\mid n
 \end{equation}
 for some $k\in \mathbb{N}_0$. In particular, if at least one such $d$ exists then $s\mid n$. Hence, if $s\nmid n$ then we have $J_k^{j,m}(n)=0.$ Assume now that $s\mid n$ and let $n'=n/s$. By \eqref{d}, it  follows immediately that 
 $$\{d\in \mathbb{N}: d\mid n,\, d\equiv j\pmod m\}=\{sd' \in \mathbb{N}:  d'\mid n',\, d'\equiv j'\pmod{ m'}\}.$$ Then
 $$J_k^{j,m}(n)=\sum_{\substack{d'\mid n'\\ d'\equiv j'\pmod{ m'}}} (sd')^k\mu(n'/d')=s^kJ_k^{j',m'}(n').$$

 \subsection{Compositions and partitions}
Let $k\in \mathbb{N}$. A $k$-{\it composition} of $n\in\mathbb{N}$ is an ordered $k$-tuple $x=(x_1,\dots,x_k)$ where, for every
$j\in [k]$, $x_j\in\mathbb{N}$ and $\sum_{j=1}^{k}x_j=n.$ 
Let $c_k(n)$ be the number of $k$-compositions of $n$.
Then $c_k(n)=0$ for all $n<k$ and it is well known that, for every $n\geq k$, we have
\begin{equation}\label{compositions}
c_k(n)=\binom{n-1}{k-1}=\frac{(n-1)\cdots (n-k+1)}{(k-1)!}.
\end{equation}
Consider now the corresponding polynomial
\begin{equation*}\label{pol-compositions}
C_k(X):=\frac{(X-1)\cdots (X-k+1)}{(k-1)!}=\sum_{i=0}^{k-1}a_{ki}X^i\in \Q[X],
\end{equation*}
and note that
 $c_k(n)=C_k(n)$ holds, not only for $n\geq k$ but for all $n\geq 1$ because any positive integer less than $k$ is a root of $C_k(X)$.
Thus,
 \begin{equation}\label{cn-pol}
c_k(n)=\sum_{i=0}^{k-1}a_{k,i}n^i\quad {\rm for ~all}\quad  n\geq 1.
\end{equation}

 We call $C_k(X)$ the $k$-composition polynomial. Recalling (\cite[Definition 8.1]{Char})  that the Stirling numbers of the first kind $s(k,i)$ are given for  $1\leq i\leq k$ by
 $$X(X-1)\cdots (X-k+1)=\sum_{i=1}^ks(k,i)X^i,$$
 it immediately follows that for every $i\in [k-1]_0$ we have
 \begin{equation}\label{Stir}
 a_{ki}=\frac{s(k,i+1)}{(k-1)!}.
 \end{equation}
In particular, $a_{k,k-1}=\frac{1}{(k-1)!}$ so that
\begin{equation*}\label{compositions-asy}
c_k(n)=\frac{1}{(k-1)!}n^{k-1}+O(n^{k-2}),
\end{equation*}
with the implied constant in the Landau symbol depending on $k.$

The generating function of $c_k(n)$ is well known (\cite[Example I.6]{AC}) and given by
\begin{equation*}\label{gen-compositions}
\sum_{n\geq 1}c_k(n)z^n=\frac{z^k}{(1-z)^{k}}.
\end{equation*}
The above equality can be obviously rewritten in terms of the $k$-composition polynomial as
\begin{equation}\label{gen-compositions-pol}
\sum_{n\geq 1}C_k(n)z^n=\frac{z^k}{(1-z)^{k}}. 
\end{equation}
Since in the above sums the first $k-1$ terms are zero we deduce that
\begin{equation}\label{lemma3}
\frac{1}{(1-z)^{k}}=\sum_{n\geq0}\binom{n+k-1}{k-1}z^n=\sum_{n\geq0}C_k(n+k)z^n.
\end{equation}

A $k$-{\it partition} of $n\in\mathbb{N}$ is an unordered $k$-tuple $x=[x_1,\dots,x_k]$ where, for every
$j\in [k]$, $x_j\in\mathbb{N}$ and $n=\sum_{j=1}^{k}x_j.$ Both for compositions and for partitions $x$,
the numbers $x_1,\ldots, x_k$ are called the terms of $x$. 
Let  $ p_k(n)$ be the number of $k$-partitions of $n$. 
Again we have $p_k(n)=0$ for all $n<k$. The formulas for $p_k(n)$ for $k=2,3$ are known (see, for example, \cite[page~$81$]{Andrews}). 
The generating function of $p_k(n)$ is also well-known and given by
 (\cite[page~$45$]{AC})
\begin{equation}\label{gen-part}
\sum_{n\ge 1} p_k(n)z^n=\frac{z^k}{(1-z)(1-z^2)\cdots (1-z^k)}.
\end{equation}

 Partitions and compositions are strictly linked and in many occasions one deduces formulas from the ones starting from those for the other one. But dealing with partitions is considerably harder than dealing with compositions and formulas become more complicate.

A $k$-composition (a $k$-partition) of $n$ is called {\it coprime} provided that $\gcd(x_1,\dots,x_k)=1$ or, equivalently,  if $\gcd(x_1,\dots,x_k,n)=1.$
 We denote with $c'_k(n)$ and with  $p'_k(n)$ the number of coprime $k$-compositions and $k$-partitions of $n$ respectively. It is easy to check that
 $
 c_k(n)=\sum_{d\mid n} c_k'(n/d),
 $
 as well as
 $
 p_k(n)=\sum_{d\mid n} p_k'(n/d).
 $
 Hence, by M\"obius inversion, we also have
 \begin{equation}
 \label{Mob-comp}
c'_k(n)=\sum_{d\mid n} \mu(n/d)\,c_k(d),
\end{equation}
 and
\begin{equation}\label{Mob-part}
 p'_k(n)=\sum_{d\mid n} \mu(n/d)\,p_k(d).
\end{equation}

 \section{Coprime $k$-compositions and asymptotics}\label{cop-comp}
Since it is well known that
 $$J_k(n)=|\{(x_1,\dots,x_k)\in \mathbb{N}^k: \forall i\in [k], \ 1\leq x_i\leq n,\  \gcd(x_1,\dots,x_k,n)=1\}|,$$
the role of the Jordan totient functions in describing the number of coprime compositions or partitions  is reasonably expected. For instance, in  \cite[Theorem 1.1, Theorem 2.2]{bach} it is shown that
\begin{equation}\label{k2N}
p'_2(n)=\frac{J_1(n)}{2}\quad {\rm for~all}\quad n\geq 3,
\end{equation}
 and
 \begin{equation}\label{k3N}
 p'_3(n)=\frac{J_2(n)}{12} \quad {\rm for~all}\quad n\geq 4.
 \end{equation}

In the following proposition we describe the easy case of  compositions and determine the asymptotic behavior of both  coprime compositions and partitions. We stress that part $(i)$ and $(ii)$ are not a novelty. For instance they appear in \cite[page 2]{To}. We reprove briefly them, for completeness.

\begin{proposition}\label{cop-not} Let $k\in \mathbb{N}$ and $a_{ki}$ as in \eqref{Stir}, for $i\in [k-1]_0$. Then the following facts hold:
\begin{itemize}
\item[(i)]For every $n\geq 1$, we have $c'_k(n)=\sum_{i=0}^{k-1}a_{ki}\,J_{i}(n).$
In particular, $c'_k(n)$ is a $\Q$-linear combination of the Jordan totient functions.
\item[(ii)] For $k\geq 2$, we have
$$c'_k(n)=\frac{1}{(k-1)!}J_{k-1}(n)+O(n^{k-2}).$$
\item[(iii)] For $k\geq 2$, we have
\begin{equation}\label{as-part}
p_k(n)=\frac{1}{k! (k-1)!} n^{k-1} +O(n^{k-2})
\end{equation}
and
$$p'_k(n)=\frac{1}{k!(k-1)!}J_{k-1}(n)+O(n^{k-2}).$$
\end{itemize}
In all the above formulas all the implied constants in the Landau symbols depend on $k.$
\end{proposition}
\begin{proof}
$(i)$ By \eqref{cn-pol},  for every $n\geq 1$, we have  $c_k(n)=\sum_{i=0}^{k-1}a_{ki}n^i$. Hence, recalling the definition \eqref{Jordan} and using \eqref{Mob-comp}, we get
$$c'_k(n)= \sum_{d\mid n} \mu(n/d)\,c_k(d)=\sum_{d\mid n} \mu(n/d)\,\sum_{i=0}^{k-1}a_{ki}d^i=\sum_{i=0}^{k-1}a_{ki}J_i(n).$$

$ (ii)$ It follows immediately by $(i)$  and by \eqref{Stir}, since  ${\displaystyle{a_{k,k-1}=\frac{1}{(k-1)!}.}}$

For $(iii)$, estimate \eqref{as-part} is formula (4.3) in \cite{JLRA}. Without the estimate of the error term it is attributed to the 1926 paper of Schur \cite{Schur}. With the error term, it is attributed to Nathanson \cite{Nat}.  The estimate of $p_k(n)'$ for $k=2,3$ comes immediately from \eqref{k2N} and \eqref{k3N}. 
 For $k\ge 4$ it follows instead from \eqref{as-part} and \eqref{Mob-part} since in this case 
\begin{eqnarray*}
p_k'(n) & = & \sum_{d\mid n} \mu(d)p_k(n/d)=\frac{1}{k!(k-1)!}\sum_{d\mid n} \mu(d)(n/d)^{k-1}+O\left(\sum_{d\mid n} \left(\frac{n}{d}\right)^{k-2}\right)\\
& = & \frac{J_{k-1}(n)}{k!(k-1)!} J_{k-1}(n)+O\left(n^{k-2}\sum_{d\ge 1} \frac{1}{d^{k-2}}\right)\\
& = & \frac{1}{k!(k-1)!}J_{k-1}(n)+O(n^{k-2}).
\end{eqnarray*}
\end{proof}
The above proposition gives, among other things, an easy formula for calculating $c'_k(n)$ in terms of the prime divisors of $n$. For instance, by \eqref{compositions}, we have
\begin{equation*}\label{coco34} c_2(n)=n-1,\qquad c_3(n)=\frac{n^2-3n+2}{2}\quad \hbox{and}\quad  c_4(n)=\frac{n^3-6n^2+11n-6}{6}.
\end{equation*}
Thus from Proposition \ref{cop-not}\,$(i)$, we get for every $n\geq 2$,
\begin{equation*}\label{co2}
c'_2(n)=J_1(n),
\end{equation*}
\begin{eqnarray*}\label{co3}
c'_3(n) & = & \frac{1}{2}J_2(n)- \frac{3}{2}J_1(n)=\frac{1}{2} n^2\prod_{p\mid n} \left(1-\frac{1}{p^{2}}\right)-\frac{3}{2}n\prod_{p\mid n} \left(1-\frac{1}{p}\right),
\end{eqnarray*}
and
\begin{eqnarray*}\label{co4}
c'_4(n)&=&\frac{1}{6}J_3(n)- J_2(n)+\frac{11}{6}J_1(n)\nonumber\\
& = & \frac{1}{6}n^3\prod_{p\mid n} \left(1-\frac{1}{p^{3}}\right)-n^2\prod_{p\mid n} \left(1-\frac{1}{p^{2}}\right)+\frac{11}{6}n\prod_{p\mid n} \left(1-\frac{1}{p}\right).
\end{eqnarray*}
One can wonder if similar easy formulas could hold for partitions too, just adapting the coefficients of the Jordan totient functions.
 Formulas \eqref{k2N} and \eqref{k3N}  seem encouraging in this direction. However, in  \cite{bach}  it is observed that the situation becomes very complicated  for $k\geq 4$ and no
 information is given for the general approach. Our paper aims to explain in which sense  and why complications do arise.

Note that, since $p'(2)\neq \frac{J_1(2)}{2}$ as well as $p'_3(3)\neq \frac{J_2(3)}{12}$  the limitations on $n$ in \eqref{k2N} and in \eqref{k3N} cannot be eliminated but, in principle, one cannot exclude that the small cases for $n$ could be included in a more rich formula involving as terms other Jordan totient functions.

Inspired by \eqref{k2N}-\eqref{k3N}, we then consider four problems:

\begin{problem}\label{pb3} Determine the $k\geq 2$ such that $p'_k(n)$ is a $\C$-linear combination of the Jordan root totient functions  in the entire domain $n\geq 1$.
\end{problem}

\begin{problem}\label{pb4} Determine the $k\geq 2$ such that $p'_k(n)$ is a $\C$-linear combination of the Jordan modulo totient functions  in the entire domain $n\geq 1$.
\end{problem}
\begin{problem}\label{pb1} Determine the $k\geq 2$ such that $p'_k(n)$ is a $\C$-linear combination of the Jordan totient functions in the entire domain $n\geq 1$.
\end{problem}
\begin{problem}\label{pb2} Determine the $k\geq 2$ such that $p'_k(n)$ is a $\C$-linear combination of the Jordan totient functions in a domain $n\geq N_k$ for some suitable $N_k\in\N$ depending on $k$.
\end{problem}

\section{Sequences with rational generating functions}\label{final}

The next classical result is called the Binet formula for linear recurrences. See, for example, Theorem C.1 in \cite{ST}. The same contents appear also, with some minor further details, in \cite[Sections 4.1--4.4]{Stan}.
\begin{proposition}
\label{recurr} Let  $P(X)=\prod_{j=1}^s (1-\alpha_jX)^{b_j}\in {\mathbb C}[X]$ be a polynomial with $P(0)=1$, and distinct nonzero roots $ \alpha_1^{-1},\dots, \alpha_s^{-1}\in \C^*$,  $s\geq 1$, 
of multiplicities $b_1,\ldots,b_s$, respectively. Given $Q(X)\in {\mathbb C}[X]$ of degree smaller than ${\text{\rm deg}}(P)$ write the Taylor expansion of $Q(z)/P(z)$ as 
\begin{equation}
\label{eq:Q}
\frac{Q(z)}{P(z)}=\sum_{n\geq 0}a_n z^n\qquad {\text{for}}\quad |z|<\max_{1\le j\le s}\{|\alpha_j|\}.
\end{equation}
Then, for every $j\in [s]$, there exists uniquely determined $P_j(X)\in  \C[X]$ of degree at most $b_j-1$ such that 
\begin{equation}
\label{eq:an}
a_n=\sum_{j=1}^sP_j(n)\alpha_j^n \quad \hbox{ for all }\ n\geq 0.
\end{equation}
If $\gcd(P(X),Q(X))=1$, then $P_j(X)$ has degree exactly $b_j-1$ for $j\in [s]$. 
Conversely, if $P_j(X)\in {\mathbb C}[X]$ are polynomials of degree at most $b_j-1$ and $a_n$ is given by formula \eqref{eq:an} then formula \eqref{eq:Q} holds with 
some polynomial $Q(X)\in {\mathbb C}[X]$  of degree less than ${\text{\rm deg}}(P)$.
\end{proposition}
The data $P_j(X)$ for $j\in [s]$ can be computed in the following way. 
We focus on the case $\gcd(P(X),Q(X))=1$ which is important for our scope. Then  $b_j$ is the order of the pole $\alpha_j^{-1}$  in the rational function ${\displaystyle{\frac{Q(z)}{P(z)}}}.$ Moreover, since $P(0)=1$, we have 
\begin{equation}\label{fact}
P(z)=c_b\prod_{j=1}^s(z- \alpha_j^{-1})^{b_j}=\frac{c_b(-1)^b}{\prod_{j=1}^s\alpha_j^{b_j}}\prod_{j=1}^s(1-\alpha_jz)^{b_j}=\prod_{j=1}^s(1-\alpha_jz)^{b_j}.
\end{equation}
By partial fractions decomposition we have
\begin{equation}\label{partial-classic}
\frac{Q(z)}{P(z)}=\sum_{j=1}^s\sum_{i=1}^{b_j}\frac{s_{ji}}{(z-\alpha_j^{-1})^{i}},
\end{equation}
where
\begin{equation*}\label{sij}
s_{ji}=\frac{1}{(b_j-i)!}\lim_{z \to \alpha_j^{-1}}D_z^{b_j-i}\left(\frac{(z-\alpha_j^{-1})^{b_j}Q(z)}{P(z)}\right),
\end{equation*}
and $D_z^{\ell}$ denotes the derivation of order $\ell\in \N_0$ with respect to the variable $z$.
Observe that, for every $j\in [s]$, we have $s_{jb_j}\neq 0$ because otherwise the order of the pole $\alpha_j^{-1}$  in ${\displaystyle{\frac{Q(z)}{P(z)}}}$ would be less than $b_j.$ 

Putting
\begin{equation}\label{rij}
r_{ji}:=(-\alpha_j)^is_{ji},
\end{equation}
from \eqref{partial-classic}, \eqref{gen-compositions-pol} and \eqref{fact}  we then obtain
\begin{eqnarray*}\label{partial2}
\frac{Q(z)}{P(z)} & = & \sum_{j=1}^s\sum_{i=1}^{b_j}\sum_{n\geq 0}r_{ji}C_i(n+i)\alpha_j^nz^n=\sum_{n\geq 0}\sum_{j=1}^s\left(\sum_{i=1}^{b_j}r_{ji} C_i(n+i)\right)\alpha_j^n z^n\\
& = & \sum_{n\geq 0}\sum_{j=1}^sP_j(n)\alpha_j^n z^n,\quad {\text{\rm where}}\quad P_j(n):=\sum_{i=1}^{b_j}r_{ji}C_i(n+i)
\end{eqnarray*}
has degree exactly $b_j-1$ since $r_{jb_j}\neq 0$.
By identifying coefficients we get
$$
a_n=\sum_{j=1}^sP_j(n)\alpha_j^n\qquad {\text{\rm for~all}}\quad n\ge 0.
$$
An explicit expression for the coefficients of the polynomials $P_j$ can be obtained as follows. Fix $j\in [s].$ We have
\begin{eqnarray*}\label{explicitPj}
P_j(n)&=&\sum_{i=1}^{b_j}r_{ji}C_i(n+i)=\sum_{i=1}^{b_j}r_{ji}\sum_{\ell=0}^{i-1}a_{i\ell}(n+i)^{\ell}\\
&=& \sum_{i=1}^{b_j}r_{ji}\sum_{\ell=0}^{i-1}\sum_{t=0}^{\ell}a_{i\ell}\binom{\ell}{t}n^t i^{\ell-t}= \sum_{i=1}^{b_j}\sum_{t=0}^{i-1}\sum_{\ell=t}^{i-1}r_{ji}a_{i\ell}\binom{\ell}{t}n^t i^{\ell-t}\\
&=& \sum_{t=0}^{b_j-1}\left(\sum_{i=t+1}^{b_j}\sum_{\ell=t}^{i-1}r_{ji}a_{i\ell}\binom{\ell}{t}i^{\ell-t}\right)n^t.
\end{eqnarray*}
Thus, 
\begin{equation}\label{defuij}
P_j(n)=\sum_{t=0}^{b_j-1} u_{jt}n^t, \quad {\text{\rm where}}\quad 
u_{jt}=\sum_{i=t+1}^{b_j}\sum_{\ell=t}^{i-1}r_{ji}a_{i\ell}\binom{\ell}{t}i^{\ell-t}.
\end{equation}
Now recall that the $a_{i\ell}$ are explicitly given by \eqref{Stir} in terms of the Stirling numbers and the $r_{ji}$ are explicitly given by \eqref{rij}.

\begin{example}
\label{ex}
\begin{itemize}
\item[(i)] Let $A(X)\in {\mathbb C}[X]$ be of degree at most $t$. The sequence of general term $a_n:=A(n)$ for  $n\geq 0$  satisfies \eqref{eq:an} with $s=1$, $\alpha_1=1$ and $b_1=t+1$. Thus, 
$$
\sum_{n\ge 0} a_n z^n=\frac{Q(z)}{(1-z)^{t+1}}\qquad {\text{for}}\qquad |z|<1
$$
holds with some $Q(X)\in {\mathbb C}[X]$ of degree at most $t$.
\item[(ii)]  For $\omega\in U_m$,  the function $T_{\omega}:\N_0\rightarrow \C$ defined, for every $n\in \N_0$, by $T_{\omega}(n)=\omega^n$ are $m$ integer periodic functions of period $m$ which form a basis for the ${\mathbb C}$-vector space of periodic functions of period $m$. If $T$ is an integer periodic function of period $m\in N$, then 
there exists $Q(X)\in {\mathbb C}[X]$ with degree less than $m$ such that
\begin{equation}\label{per}
\sum_{n\ge 0} T(n) z^n=\frac{Q(z)}{1-z^m}\qquad {\text{for}}\qquad |z|<1.
\end{equation}
{\rm Indeed, in order to show that there exists a unique choice of $c_{\omega}\in \C$, for $\omega\in U_m$, such that $T=\sum _{\omega}c_{\omega}T_{\omega}$, it is sufficient to note that the linear system in the variables $(c_{\omega})_{\omega\in U_m}$
$$
T(k)=\sum_{\omega\in U_m} c_{\omega} \omega^k\qquad {\text{for}}\qquad k=0,1,\ldots,m-1
$$
has coefficient matrix given by the  invertible Vandermonde matrix with columns $(1,\omega,\ldots,\omega^{m-1})^T$ for $\omega\in U_m$. Thus, the sequence of general term $a_n=T(n)$ satisfies \eqref{eq:an} with 
$s=m$, $b_1=\cdots=b_m=1$ and $\{\alpha_1,\ldots,\alpha_m\}=U_m$. Hence, there exists $Q(X)\in {\mathbb C}[X]$ with the required properties such that
\eqref{per}  holds. }

\item[(iii)] Let $k\geq 2$. Then  $p_k(n)$ is representable as a quasi-polynomial by
\begin{equation*}\label{rep-part}
p_k(n)= P_1(n)+S(n),
\end{equation*}
where $P_1(X)\in \Q[X]$ has degree $k-1$ and $S(n)$ is a quasi--polynomial of degree $\lfloor k/2\rfloor-1$ and quasi--period $\delta(k)$. $P_1$ is called the polynomial part of $p_k$.

{\rm By \eqref{gen-part}, we have
$$\sum_{n\ge 1} p_k(n)z^n=\frac{z^k}{(1-z)(1-z^2)\cdots (1-z^k)},$$
so that Proposition \ref{recurr} applies with 
$$
P(z)=(1-z)(1-z^2)\cdots (1-z^k)\quad  \hbox{and}\quad Q(z)=z^k,
$$
which are coprime. The roots of $P$ are the elements of $V=\bigcup_{m=1}^k U_m$ and, since $V$ is closed under inversion, we have that $V$ coincides with the set  of the inverses of the roots of $P$. Let $s=|V|$. Note that  $V$ is expressible as the disjoint union $V=\bigsqcup_{m=1}^k\{\omega\in U_m: \omega\  \hbox{is primitive}\}$. Thus, since $k\geq 2$, we have $s=\sum_{m=1}^k\phi(m)\geq 2.$ We  order its elements $\omega_1,\dots, \omega_s$ considering the primitive $m$-roots of unity starting from $m=1$ and finishing with $m=k.$ Hence $\omega_1=1,\omega_2=-1.$ Let $b_j\geq 1$ be the multiplicity of $\omega_j$, for $j\in [s].$ Clearly if $\omega_j$ is a primitive $m$-root of unity, for some $m\in [k]$, we have that $b_j=\lfloor k/m\rfloor$.  In particular, $b_1=k$, $b_2=\lfloor k/2\rfloor$.
Then we have
\begin{equation}\label{part-dec}
p_k(n)=\sum_{j=1}^sP_j(n)\omega_j^n\quad  {\text{\rm for all}}\quad n\geq 1
\end{equation}
by the explicitly computable $P_j(X)\in  \C[X]$ of degree $b_j-1$, $j\in [s]$, given in \eqref{defuij}. Thus  $p_k$ is a quasi--polynomial and
the expression \eqref{part-dec} can be split into 
\begin{equation*}\label{part-dec2}
p_k(n)=P_1(n)+\sum_{j=2}^sP_j(n)\omega_j^n\quad  {\text{\rm for all}}\quad n\geq 1.
\end{equation*}
Define then $S(n):=\sum_{j=2}^sP_j(n)\omega_j^n$. By Proposition  \ref{recurr} and the above remarks we know that its degree is $\lfloor k/2\rfloor-1$ because $\omega_2$ has multiplicity $\lfloor k/2\rfloor$  and the multiplicities of the remaining roots  $\omega_j$ for $j\ge 3$ are at most $\lfloor k/2\rfloor$. The statement about the period is also clear.}

\item[(iv)]  If $k\ge 2$, there exists no $P(X)\in \mathbb{C}[X]$ such that
\begin{equation}
\label{eq:consequence0}
p_k(n)=P(n)\quad {for~all}\quad n\geq 1.
\end{equation}
Furthermore, if $k\ge 4$, there exists no $P(X)\in \mathbb{C}[X]$ and no integer periodic function $T$  such that
\begin{equation}
\label{eq:consequence}
p_k(n)=P(n)+T(n)\quad {for~all}\quad n\geq 1.
\end{equation}
\end{itemize}
{\rm Indeed, by $(i)$ and $(iii)$, the generating function of the sequence appearing in the right--hand side of \eqref{eq:consequence} is of the form
\begin{equation}
\label{eq:RHS}
\frac{Q(z)}{(1-z)^{t+1}}+\frac{R(z)}{1-z^m},
\end{equation}
whereas the generating function of the sequence appearing in the left--hand side is 
\begin{equation}\
\label{eq:LHS}
\frac{z^k}{(1-z)(1-z^2)\cdots (1-z^k)}.
\end{equation}
For $k\ge 4$, the rational function appearing at \eqref{eq:LHS} has $z=-1$ as a pole of multiplicity $\lfloor k/2\rfloor\ge 2$, which is not the case for the rational function indicated at \eqref{eq:RHS}, so equality \eqref{eq:consequence} is impossible. The fact that \eqref{eq:consequence0} is also impossible for $k=2,3$ is also immediate. }
\end{example}

\section{Coprime partitions and Jordan functions }\label{cop-part-Jord}

We are now ready to solve Problems \ref{pb3}--\ref{pb2}.  
\begin{theorem}\label{comb-genera} Let $k\geq 2$. Then the following facts hold:
\begin{itemize}
\item[$(i)$] $p_k'$ is a $\C$-linear combination of the Jordan root totient functions in the entire domain $n\geq 1$.
\item[$(ii)$] $p_k'$ is a $\C$-linear combination of the Jordan modulo totient functions in the entire domain $n\geq 1$.
\item[$(iii)$] $p_k'(n)$ is not a $\mathbb{C}$-linear combination of the Jordan totient functions in the entire domain $n\geq 1$.
\item[$(iv)$]  If $p'_k(n)$ is a $\C$-linear combination of the Jordan totient functions in a domain $n\geq N_k$ for some suitable $N_k\in\N$ depending on $k$, then $p_k'(n)$ is a rational multiple of $J_{k-1}(n)$ and $k\in \{2,3\}.$
Further, the minimal value of $N_k$ is $k+1$ for both $k\in \{2,3\}$. 
\end{itemize}
\end{theorem}
In both cases $(i)$ and $(ii)$ above the coefficients of the linear combinations are easily computable. 
\begin{proof} In Example \ref{ex}\,$(iii)$ we have seen that
\begin{equation}\label{part-dec10}
p_k(n)=\sum_{j=1}^sP_j(n)\omega_j^n\quad  {\text{\rm for all}}\quad n\geq 1
\end{equation}
for suitable explicitly computable $P_j(X)\in  \C[X]$ of degree $b_j-1$, $j\in [s]$ and $\{\omega_1,\dots, \omega_s\}=\bigcup_{m=1}^k U_m$.

$(i)$ By  \eqref{Mob-part}, \eqref{part-dec10} and \eqref{defuij}, we have
\begin{eqnarray}\label{prelim}\nonumber
p'_k(n)&=&\sum_{d\mid n}\sum_{j=1}^sP_j(d)\omega_j^d\mu(n/d)=\sum_{d\mid n}\sum_{j=1}^s\sum_{t=0}^{b_j-1} u_{jt}d^t\omega_j^d\mu(n/d)\\
&=&\sum_{j=1}^s \sum_{t=0}^{b_j-1}u_{jt} \sum_{d\mid n}\omega_j^dd^t\mu(n/d)=\sum_{j=1}^s \sum_{t=0}^{b_j-1}u_{jt} J_{(t,\omega_j)}(n).
\end{eqnarray}
which expresses explicitly $p'_k$  as a $\C$-linear combination of Jordan root totient functions.

$(ii)$ It follows immediately by $(i)$ and \eqref{root-modulo}.

$(iii)$ Assume the contrary. Then there exist $s\in\mathbb{N}$ and  $c_i\in \mathbb{C}$ for all $i=0,\dots, s$, such that
$$
p_k'(n)=\sum_{i=0}^{s} c_iJ_i(n) \quad  {\text{\rm for~all}}\quad n\geq 1.
$$
Writing the above relation for all $d\mid n$ and using \eqref{Mob-part}, we then get  for every $n\geq 1$,
\begin{equation*}\label{pkpol}
p_k(n)=\sum_{d\mid n} p_k'(d)=\sum_{d\mid n}\sum_{i=0}^{s} c_iJ_i(d)=\sum_{i=0}^{s} c_i\sum_{d\mid n}J_i(d)=\sum_{i=0}^{s} c_i n^i=P(n),
\end{equation*}
where $P(X)=\sum_{i=0}^{s} c_i X^i\in \mathbb{C}[X],$ against Example \ref{ex}\,$(iv)$.

$(iv)$  Let $N_k\in \mathbb{N}$ be minimum  such that, for $n\geq N_k$, $p_k'(n)$  is a $\mathbb{C}$-linear combination of the Jordan totient functions.
For shortness we set $N:=N_k.$
Surely $ p_k'(n)$ cannot be, for sufficiently large $n$, a multiple of $J_0(n).$ Thus, there exist
 $s\in\mathbb{N}$ and  $c_i\in \mathbb{C}$ for $i\in[ s]_0$  with $c_s\neq 0$, such that
 \begin{equation}\label{kN}
 p_k'(n)=\sum_{i=0}^{s}c_iJ_i(n) \quad {\text{\rm for~all}}\quad n\geq N.
 \end{equation}

As a consequence of $(iii)$ above we have that $N\geq 2$. Thus if $n\geq N$, we also have $n\geq 2$ and so  $J_0(n)=0$. Hence,  whatever $c_0$ is in \eqref{kN},
we can surely  guarantee the same equality adopting $c_0=0.$ Let then
\begin{equation}\label{kN2}
 p_k'(n)=\sum_{i=1}^{s}c_iJ_i(n) \quad {\text {\rm for~all}}\quad n\geq N,
 \end{equation}
and  $P(X)=\sum_{i=1}^{s} c_i X^i\in \mathbb{C}[X]$ be the corresponding polynomial. Note that $\mathrm{deg}(P)=s.$

Define the function $f_N^k:\N_0\rightarrow \C$ given for every $n\in \N$ by

 \begin{equation}\label{fNk}
 f_N^k(n)=\sum_{\substack{d\mid n\\ d<N}}\left( p_k'(d)-\sum_{i=1}^{s} c_iJ_i(d)\right)
 \end{equation}
and by $f_N^k(0)=f_N^k(m)$, where $m:=\delta(N-1).$
 
We claim that
 \begin{equation}\label{fNkeq}
 p_k(n)=P(n)+f_N^k(n) \quad {\text {\rm for~all}}\quad n\geq 1.
 \end{equation}
Indeed, by \eqref{Mob-part} and \eqref{kN2}, for every $n\geq 1$, we have
\begin{eqnarray*}
 p_k(n)& = &\sum_{d\mid n} p_k'(d)=\sum_{\substack{d\mid n\\ d<N}}p_k'(d)+\sum_{\substack{d\mid n\\ d\geq N}}p_k'(d)\\
 & = & \sum_{\substack{d\mid n\\ d<N}}p_k'(d)+\sum_{\substack{d\mid n\\ d\geq N}}\sum_{i=1}^{s}c_iJ_i(d)\\
% =\sum_{\substack{d\mid n\\ d<N}}p_k'(d)+\sum_{i=1}^{s}c_i\sum_{\substack{d\mid n\\ d\geq N}}J_i(d) \\
& = & \sum_{\substack{d\mid n\\ d<N}}\left(p_k'(d)-\sum_{i=1}^{s}c_iJ_i(d)\right)+\sum_{d\mid n}\sum_{i=1}^{s}c_iJ_i(d)\\
&=&  f_N^k(n)+\sum_{i=1}^{s}c_in^i= f_N^k(n)+P(n).
\end{eqnarray*}

 We next claim that
 \begin{equation}\label{f-pe}
  f_N^k\ \hbox{ is an integer  periodic function.}
  \end{equation}
In order to prove that it is enough to show that $m\in M( f_N^k)$. Let $\ell,n\in \N$. Since every $d<N$ divides $m$, we have
  \begin{eqnarray*}
 f_N^k(n+\ell m)& =&\sum_{\substack{d\mid n+\ell m\\ d<N}}\left(p_k'(d)-\sum_{i=1}^{s}c_iJ_i(d)\right) \\
&=&\sum_{\substack{d\mid n\\ d<N}}\left(p_k'(d)-\sum_{i=1}^{s}c_iJ_i(d)\right)
= f_N^k(n).
\end{eqnarray*}
Hence, trivially  we also have $f_N^k(0+\ell m)=f_N^k(m)=f_N^k(0).$

By \eqref{fNkeq} and \eqref{f-pe} we then have that $p_k$ is the sum of a polynomial and of  an integer periodic function. By Example \ref{ex}\,$(iv)$ this rules out $k\geq 4$, so that $k\in \{2,3\}.$

Now, by \eqref{as-part} and by \eqref{fNkeq}, we get
$$
f_N^k(n)+P(n)=\frac{1}{k! (k-1)!} n^{k-1} +O(n^{k-2}).
$$
By \eqref{f-pe},  ${\displaystyle{\frac{ f_N^k(n)}{n^{k-1}}}}$ tends to $0$ as $n$ goes to infinity. Thus,  ${\displaystyle{\frac{P(n) }{n^{k-1}}}}$ tends to ${\displaystyle{\frac{1}{k!(k-1)!}}}$ as $n$ goes to infinity, which implies
$s=\deg(P)=k-1$ and $c_{k-1}=\frac{1}{k!(k-1)!}.$
In particular, we have $P(X)=\sum_{i=1}^{k-1} c_i X^i.$

If $k=2$, this gives  $P(X)=\frac{X}{2}$ and \eqref{kN2} becomes
\begin{equation}\label{Q}
p_2'(n)=\frac{J_1(n)}{2}\  \hbox{ for~all} \ n\geq N,
\end{equation}
 which surely does not hold for $N=2$, because $p_2'(2)=1\neq 1/2.$
We know from \eqref{k2N} that
 \begin{equation}\label{R}
 p_2'(n)=\frac{J_1(n)}{2}\  \hbox{ for~all} \ n\geq 3.
 \end{equation}
  Hence, the minimum $N$ such that there exists an expression of $p_2'(n)$ as a $\C$-linear combination of the Jordan totient functions for $n\geq N$ is $N=3$ and no other such expression with $N=3$ is possible besides \eqref{R}.

If $k= 3,$ we then get $P(X)=c_1X+\frac{X^2}{12}$  and \eqref{kN2} becomes
\begin{equation}\label{A}
p_3'(n)=c_1J_1(n)+\frac{J_2(n)}{12}\ \hbox{ for~all} \  n\geq N.
\end{equation}
Assume that  $N=3$. Then, by \eqref{A}, $p_3'(3)=1$ implies $c_1=1/6$ while
$p_3'(4)=1$ implies $c_1=0,$ a contradiction. It follows that $N\geq 4.$  By \eqref{k3N}, we know that
  \begin{equation}\label{B}
  p_3'(n)=\frac{J_2(n)}{12}\ \hbox{ for~all} \  n\geq 4.
  \end{equation}
Thus, the minimum $N$ such that there exists an expression of $p_3'(n)$ as a $\C$-linear combination of the Jordan totient functions for $n\geq N$ is $N=4.$
We finally observe that no other such expression with $N=4$ is possible besides \eqref{B}. Indeed, as previously observed, the computation of $p_3'(4)$ by \eqref{A} implies $c_1=0.$
\end{proof}

\section{Computation of some generalized Jordan totient functions}\label{sum-mu}

In the last two sections of the paper we illustrate how to explicitly find the polynomials $P_j$ of \eqref{part-dec} relying on the generating function of $p_k(n)$. This allows  us
 to represent $p_k'$ as a $\C$-linear combination of Jordan root functions.
Next we explain how to explicitly compute the Jordan modulo totient functions in which those Jordan root totient functions split, making use of the Jordan-Dirichlet totient functions. We limit ourselves to treat $k\in \{2,3,4\}$. Anyway the general method should be clear.

In this section, we  gather together all the computations which we will need later. They illustrate very well how to connect the diverse generalizations of the Jordan totient functions in order to obtain one from the other.
For this reason they are of interest in themselves.  In the next section, we examine separately $p_2', p_3'$ and $p'_4$.

\begin{lemma}
\label{lem:1} Let $n\in\mathbb{N}$ and write  $n=3^bm_1$ with $\gcd(3,m_1)=1$. Then
\begin{equation}\label{31}
J_0^{1,3}(n)=\left\{\begin{matrix}
1& {\text{\rm  if}}& n=1;\\
-1& {\text{\rm  if}}& n=3;\\
0 & {\text{\rm  if}}& \exists~ p\equiv 1\pmod 3,~p\mid m_1;\\
0 & {\text{\rm  if}} & b\geq 2;\\
(-1)^{\Omega(n)} 2^{\omega(m_1)-1}& {\text{\rm  }}& {\text{\rm otherwise}}.\\
\end{matrix}\right.
\end{equation}

\end{lemma}

\begin{proof}  For shortness, write $f(n)$ instead of $J_0^{1,3}(n).$ Then
$$
f(n)=\sum_{\substack{d\mid n\\ d\equiv 1\pmod 3}} \mu(n/d).
$$
If $d\equiv 1\pmod 3$ and $d\mid n$, then $d\mid m_1$. Thus, $3^b\mid n/d$ over all such divisors $d$ and $n/d=3^b(m_1/d)$ with $3^b$ and $m_1/d$ coprime. Thus, by the multiplicativity of the $\mu$ function, we get
\begin{equation}\label{equa1}
f(n)=\sum_{\substack{d\mid n\\ d\equiv 1\pmod 3}} \mu(n/d)=\sum_{\substack{d\mid m_1\\ d\equiv 1\pmod 3}} \mu(3^b)\mu(m_1/d)=\mu(3^b)f(m_1).
\end{equation}
Hence, if $b\geq 2$ we have $f(n)=0$. Let then $b\in\{0,1\}$. By \eqref{equa1}, it suffices to study $f(m_1)$.
 Let $\chi$ be the unique non principal Dirichlet character modulo $3$. Then $\chi(k)=1$ if $k\equiv 1\pmod 3$,
$\chi(k)=-1$ if $k\equiv 2\pmod 3$ and $\chi(k)=0$ if $\gcd(k,3)>1$.  It is easily  seen that
\begin{equation*}
f(m_1)=\frac{1}{2}\sum_{d\mid m_1} (\chi(d)+1)\mu(m_1/d)=\frac{1}{2}\sum_{d\mid m_1} \chi(d)\mu(m_1/d)+\frac{1}{2}\sum_{d\mid m_1} \mu(m_1/d).
\end{equation*}
Since $m_1$ is coprime to $3$, by \eqref{Jord-Dir-formula}, we get for $m_1>1$
\begin{eqnarray}\label{equa2}\nonumber
f(m_1)&=&\frac{1}{2}\sum_{d\mid m_1} \chi(d)\mu(m_1/d)=\frac{1}{2}\chi(m_1)\prod_{p\mid m_1}\left(1-\frac{1}{\chi(p)}\right)\\\nonumber
&=&\frac{1}{2}\chi(m_1)\prod_{\substack{p\mid m_1 \\ p\equiv 1\pmod 3}}(1-1)\prod_{\substack{p\mid m_1 \\ p\equiv 2\pmod 3}}(1+1)\\\nonumber
&=& \left\{\begin{matrix} 0 & {\text{\rm if}} & p\mid m_1~  {\text{\rm for some}}~ p\equiv 1\pmod 3;\\
 (-1)^{\Omega(m_1)}2^{\omega(m_1)-1} & {\text{\rm if}} & p\equiv 2\pmod 3 ~  {\text{\rm for all}}~ p\mid m_1.
\end{matrix}\right.
\end{eqnarray}
Thus, by \eqref{equa1}, taking into consideration that $f(1)=1$,
 the formula \eqref{31} for $f(n)$ immediately follows.
\end{proof}

\begin{lemma}\label{comp}  Let $n\in\mathbb{N}$. 
\begin{itemize}
\item[$(i)$] Then
$$
J_{(0,-1)}(n)=\left\{\begin{matrix} -1 & {\text{\rm if}} & n=1;\\
2 & {\text{\rm if}} & n=2;\\
0 & {\text{\rm  if}} & n>2.\end{matrix}\right.
$$
\item[$(ii)$] Write $n=2^am$ with $m$ odd. Then
$$
J_{(1,-1)}(n)=\left\{ \begin{matrix} -\phi(n) & {\text{\rm if}} & a=0;\\
3\phi(n) & {\text{\rm if}} & a=1;\\
\phi(n) & {\text{\rm if}} & a\ge 2.\end{matrix}
\right.
$$
\item[$(iii)$] Write $n=2^am$ with $m$ odd. Then
$$
J_{(0,i^k)}(n)=\left\{\begin{matrix} i^k & {\text{\rm if}} & n=1;\\
-1-i^k & {\text{\rm if}} & n=2;\\
0 & {\text{\rm if}} & \exists~p\equiv 1\pmod 4,~p\mid m; \\
2 & {\text{\rm if}} & n=4;\\
0 & {\text{\rm if}} & a\ge 3 ~{\text{\rm or }}  a= 2~{\text{\rm and}} ~m>1;\\
i^k(-1)^{\Omega(n)}2^{\omega(m)} & {\text{\rm }} & {\text{\rm otherwise}},\\
\end{matrix}\right.
$$
for $k=1,3$.
\item[$(iv)$] Write $n=3^bm_1$ with $\gcd(m_1,3)=1$, and denote by $\omega$ the principal $3$-root of $1$. Then
$$
J_{(0,\omega^k)}(n)=\left\{\begin{matrix} \omega^k & {\text{\rm if}} & n=1;\\
 -\omega^k+1 & {\text{\rm if}} & n=3;\\
 0 & {\text{\rm if}} & \exists~ p\equiv 1\pmod 3,~p\mid m_1;\\
 0 & {\text{\rm if}} & b\ge 2;\\
  (\omega^k-\omega^{2k}) (-1)^{\Omega(n)}2^{\omega(m_1)-1} & {\text{\rm }} & {\text{\rm otherwise}},\\
 \end{matrix}\right.
$$
for $k=1,2$.
\end{itemize}
\end{lemma}

\begin{proof}
$(i)$ Write $n=2^am$ with  $m$ odd. For $n=1$ and $n=2$ one makes a direct computation. For $n\ge 3$, note that
\begin{equation}
\label{eq:1}
\sum_{\substack{d\mid n\\ d~{\text{\rm odd}}}}\mu(n/d)=\sum_{d\mid m}\mu(2^a (m/d))=\mu(2^a)\sum_{d\mid m} \mu(m/d).
\end{equation}
If $m=1$, then $a\geq 2$ and thus $\mu(2^a)=0$ so that, by \eqref{eq:1}, we get
$
\displaystyle{\sum_{\substack{d\mid n\\ d~{\text{\rm odd}}}}\mu(n/d)=0.}
$
If $m>1$, then $m\geq 3$ so that
$\displaystyle{\sum_{d\mid m} \mu(m/d)=0}
$
and, by \eqref{eq:1}, we again get
$
\sum_{\substack{d\mid n\\ d~{\text{\rm odd}}}}\mu(n/d)=0.
$
It follows that
$$
0=\sum_{d\mid n} \mu(n/d)=\sum_{\substack{d\mid n\\ d~{\text{\rm even}}}}\mu(n/d)+\sum_{\substack{d\mid n\\ d~{\text{\rm  odd}}}} \mu(n/d)=\sum_{\substack{d\mid n\\ d~{\text{\rm even}}}}\mu(n/d).
$$
Hence,
$$
J_{(0,-1)}(n)=\sum_{d\mid n} (-1)^d\mu(n/d)=-\sum_{\substack{d\mid n\\ d~{\text{\rm odd}}}}\mu(n/d)+\sum_{\substack{d\mid n\\ d~{\text{\rm even}}}}\mu(n/d)=0.
$$
$(ii)$ We start again with the odd $d$'s getting
\begin{equation}
\label{eq:2}
\sum_{\substack{d\mid n\\ d~{\text{\rm odd}}}} d\mu(n/d)=\sum_{d\mid m}d\mu(2^a (m/d))=\mu(2^a)\sum_{d\mid m}d\mu(m/d)=\mu(2^a)\phi(m).
\end{equation}
The above calculation proves $(ii)$ if $a=0$. If $a\ge 2$, the right--hand side above is zero. Hence, we have
$$
\phi(n)=\sum_{d\mid n} d\mu(n/d)=\sum_{\substack{d\mid n\\ d~{\text{\rm even}}}} d\mu(n/d)+\sum_{\substack{d\mid n\\ d~{\text{\rm odd}}}} d\mu(n/d)=\sum_{\substack{d\mid n\\ d~{\text{\rm even}}}} d\mu(n/d),
$$
 so that we also have
$$
J_{(1,-1)}(n)=\sum_{d\mid n} (-1)^d d\mu(n/d)=\sum_{\substack{d\mid n\\ d~{\text{\rm even}}}} d\mu(n/d)-\sum_{\substack{d\mid n\\ d~{\text{\rm  odd}}}}d \mu(n/d)=\phi(n).
$$
Finally, if $a=1$, we have $n=2m$ and then every even divisor of $2m$ is of the form $2d$ for $d\mid m$. Thus,
\begin{eqnarray*}
J_{(1,-1)}(n) & = & \sum_{\substack{d\mid 2m\\ d~{\text{\rm even}}}}d\mu(2m/d)-\sum_{\substack{d\mid 2m\\ d~{\text{\rm odd}}}} d\mu(2m/d)\\
&=& \sum_{d\mid m}  (2d)\mu(2m/2d)+\sum_{d\mid m} d\mu(m/d)\\
& = & 2\sum_{d\mid m}d\mu(m/d)+\phi(m)=3\phi(m)= 3\phi(n).\\
\end{eqnarray*}

$(iii)$ The function $f_k(n)=i^{k(n-1)}$ defined for odd $n$ and extended to all positive integers by putting $f_k(n)=0$ for even $n$, is totally multiplicative. Indeed, if $m,n$ are both odd, we then then have
$f_k(mn)=i^{k(mn-1)}$ and $f_k(m)f_k(n)=i^{k(m-1)}i^{k(n-1)}=i^{k(m+n-2)}$ and then the equality
$$f_k(mn)=f_k(m)f_k(n)$$ is equivalent to
$$
i^{k(mn-1)}=i^{k(m+n-2)},
$$
which is equivalent to
$$
1=i^{k(mn-m-n+1)}=i^{k(m-1)(n-1)},
$$
which holds since both $m-1$ and $n-1$ are even. If instead at least one of $m$ and $n$ is even, then $mn$ is even, so that $f_k(mn)=0=f_k(m)f_k(n).$

So,
\begin{equation}
\label{eq:3}
\sum_{\substack{d\mid n\\ d~{\text{\rm odd}}}} i^{kd}\mu(n/d)=i^k\mu(2^a)\sum_{d\mid m} i^{k(d-1)}\mu(m/d)=i^k\mu(2^a)(f_k*\mu)(m),
\end{equation}
and $f_k*\mu$ is multiplicative. If $m=p^{\lambda}$, with $p$ an odd prime and $\lambda \geq 1$,
 then
\begin{eqnarray*}
(f_k*\mu)(p^{\lambda}) & = & \sum_{d\mid p^{\lambda}} i^{k(d-1)}\mu(p^{\lambda}/d)=-i^{k(p^{\lambda-1}-1)}+i^{k(p^{\lambda}-1)}\\
& = & \left\{\begin{matrix} 0 & {\text{\rm if}} & p\equiv 1\pmod 4;\\
  1+(-1)^{k+1} & {\text{\rm if}} & p\equiv 3\pmod 4,~2\mid \lambda;\\
 (-1)( 1+(-1)^{k+1}) & {\text{\rm if}} & p\equiv 3\pmod 4,~2\, {\nmid }\, \lambda.
 \end{matrix}\right.
 \end{eqnarray*}

So, if $k=1,3$, we get that $(f_k*\mu)(p^{\lambda})$ equals $0$ when $p\equiv 1\pmod 4$ and equals $2(-1)^{\lambda}$ if $p\equiv 3\pmod 4$. We thus get that for $m>1$,
\begin{equation}\label{star}
(f_k*\mu)(m)=\left\{\begin{matrix}
0 & {\text{\rm if}} & p\equiv 1\pmod 4~{\text{\rm for~some}} ~p\mid m;\\
(-1)^{\Omega(m)} 2^{\omega(m)} & {\text{\rm if}} & p\equiv 3\pmod 4~{\text{\rm for ~all}}~p\mid m.
\end{matrix}\right.
\end{equation}
If $a=0$, then $n=m$ is odd and thus, by \eqref{eq:3}, we get
$$
J_{(0,i^k)}(n)=\sum_{d\mid n} i^{kd}\mu(n/d)=i^k(f_k*\mu)(m),
$$
and this is $i^k$ if $n=m=1$, $0$ if $m>1$ and $p\mid m$ for some prime number $p\equiv 1\pmod 4$ and $i^k(-1)^{\Omega(m)}2^{\omega(m)}=i^k (-1)^{\Omega(n)}2^{\omega(m)}$, otherwise.

 If $a\ge 2$, then by \eqref{eq:3}, the sum over the divisors $d$ of $n$ which are odd is zero since $\mu(2^a)=0$. Thus, the given sum is concentrated on the even divisors and we have
\begin{eqnarray*}
J_{(0,i^k)}(n)&=&\sum_{2d\mid n} i^{k(2d)}\mu(n/2d)=\sum_{d\mid n/2} (-1)^{kd} \mu((n/2)/d)\\
&=&\sum_{d\mid n/2} (-1)^d\mu((n/2)/d)
\end{eqnarray*}
for $k=1,3$. Moreover, by $(i)$ and the fact that $n/2\ge 2$, this last sum is zero unless $n/2=2$ in which case it is $2$.

Let finally $a=1$, so that $n=2m.$ We compute that the given sum is $-1-i^k$ for
$n=2$. Now assume $m>1$. In this case, by \eqref{eq:3}, we have
$$
\sum_{\substack{d\mid n\\ d~{\text{\rm odd}}}} i^{kd}\mu(n/d)=-i^k (f_k*\mu)(m)
$$
and, by \eqref{star}, this is zero unless all prime factors of $m$ are congruent to $3$ modulo $4$ in which case it is $-i^k(-1)^{\Omega(m)}2^{\omega(m)}=i^k (-1)^{\Omega(n)}2^{\omega(m)}$.
As for the even divisors, these are of the form $2d$ for some $d\mid m$, and we get
$$
\sum_{\substack{d\mid n\\ d~{\text{\rm even}}}} i^{kd}\mu(n/d)=\sum_{d\mid m} i^{k(2d)}\mu(2m/2d)=\sum_{d\mid m} (-1)^d\mu(m/d),
$$
and, by $(i)$, this last sum is $0$ since $m\geq 3$.

$(iv)$ We have
\begin{eqnarray}\label{terms}
\nonumber J_{(0,\omega^k)}(n)&=&\sum_{d\mid n} \omega^{dk}\mu(n/d) \\
& = &   \omega^k\sum_{\substack{d\mid n\\ d\equiv 1\pmod 3}} \mu(n/d)+\omega^{2k}\sum_{\substack{d\mid n\\ d\equiv 2\pmod 3}} \mu(n/d)+\sum_{\substack{d\mid n\\ 3\mid d}} \mu(n/d)\nonumber\\
& = & \omega^kS_1(n)+\omega^{2k}S_2(n)+S_0(n),
\end{eqnarray}
where, for shortness, we have set $S_j(n):=J_{0}^{j,3}(n),$  for $j\in \{0,1,2\}.$
Thus, we need to compute $S_j(n)$, for $j\in \{0,1,2\}.$

The easiest one is $S_0$. If $3\nmid n$, we obviously have that $S_0(n)=0$. If $3\mid n$, that is $b\geq 1$, we instead have, by \eqref{Jord0}:
$$
S_0(n)=\sum_{\substack{d\mid n\\ 3\mid d}} \mu(n/d)=\sum_{d\mid n/3} \mu((n/3)/d)=\left\{\begin{matrix} 1 & {\text{\rm if}} & n=3;\\ 0 &{\text{\rm if}} & n>3.\end{matrix}\right.
$$
So, $S_0(n)$ is always $0$ except if $n=3$ when it is $1$. As for $S_1,~S_2$, we write
$$
S_1(n)=\sum_{\substack{d\mid n\\ d\equiv 1\pmod 3}} \mu(n/d)=\mu(3^b)\sum_{\substack{d\mid m_1\\ d\equiv 1\pmod 3}} \mu(m_1/d)=\mu(3^b)S_1(m_1),
$$
and similarly $S_2(n)=\mu(3^b)S_2(m_1)$. By \eqref{31}, 
we have
$$
S_1(m_1)=\left\{\begin{matrix} 1 & {\text{\rm if}} & m_1=1;\\
 0 & {\text{\rm if}} & \exists~p\equiv 1\pmod 3,~p\mid m_1;\\
 (-1)^{\Omega(m_1)} 2^{\omega(m_1)-1} & {\text{\rm }} & {\text{\rm otherwise}}.
 \end{matrix}\right.
 $$
 Since
 $$
 S_1(m_1)+S_2(m_1)=\sum_{d\mid m_1} \mu(m_1/d)
 $$
 is $1$ for $m_1=1$ and $0$ for $m_1>1$, we get that
 $$
 S_2(m_1)=\left\{\begin{matrix} 0 & {\text{\rm if}} & m_1=1;\\
 0 & {\text{\rm if}} & \exists ~p\equiv 1\pmod 3,~p\mid m_1;\\
 -(-1)^{\Omega(m_1)} 2^{\omega(m_1)-1} & {\text{\rm }} & {\text{\rm otherwise}}.
 \end{matrix}\right.
 $$
Thus, by \eqref{terms}, get that
 $$
J_{(0,\omega^k)}(n)=\left\{\begin{matrix} \omega^k & {\text{\rm if}} & n=1;\\
 -\omega^k+1 & {\text{\rm if}} & n=3;\\
 0 & {\text{\rm if}} & \exists~ p\equiv 1\pmod 3,~p\mid m_1;\\
 0 & {\text{\rm if}} & b\ge 2;\\
  (\omega^k-\omega^{2k}) (-1)^{\Omega(n)}2^{\omega(m_1)-1} & {\text{\rm }} & {\text{\rm otherwise}}.\\
 \end{matrix}\right.
 $$
\end{proof}

\section{Partitions and coprime partitions into $k$ parts for $k\in \{2,3,4\}$}\label{small}

\subsection{The case of $2$ parts}\label{small2}
By \eqref{gen-part}, we have
$$
\sum_{n\ge 1} p_2(n)z^n=\frac{z^2}{(1-z)(1-z^2)}.
$$
Partial fraction expansion gives
\begin{eqnarray*}
\frac{z^2}{(1-z)(1-z^2)} & = & \frac{z^2}{(1-z)^2(1+z)}= \frac{-3}{4(1-z)}+ \frac{1}{2(1-z)^2}+ \frac{1}{4(1+z)}.\\
\end{eqnarray*}
Hence, using formula \eqref{lemma3}, we get
\begin{eqnarray*}
\sum_{n\ge 1} p_2(n)z^n&=&\frac{-3}{4}\sum_{n\geq 0}z^n+\frac{1}{2}\sum_{n\geq 0}(n+1)z^n+\frac{1}{4}\sum_{n\geq 0}(-1)^nz^n\\
&=& \sum_{n\geq 0}\left(\frac{2n-1}{4}+\frac{(-1)^n}{4}\right)z^n
\end{eqnarray*}
and thus
\begin{equation}\label{p2}
p_2(n)=\frac{2n-1}{4}+\frac{(-1)^n}{4}.
\end{equation}
This is, of course, a reedition of the obvious $p_2(n)=\lfloor\frac{n}{2}\rfloor$, which puts in evidence the nature of $p_2(n)$ as a sum of a polynomial and of a periodic function of period $2.$
By \eqref{Mob-part}, we then get for every $n\geq 1$
\begin{eqnarray}\label{p'2}
\nonumber p'_2(n)&=&\frac{1}{2}\sum_{d\mid n}d\mu(n/d)-\frac{1}{4}\sum_{d\mid n}\mu(n/d)+\frac{1}{4}\sum_{d\mid n}(-1)^d\mu(n/d)\\
&=&\frac{1}{2}J_1(n)-\frac{1}{4}J_0(n)+\frac{1}{4}J_{(0,-1)}(n).
\end{eqnarray}
By Lemma \ref{comp}\,$(i)$,
$$\frac{1}{4}J_{(0,-1)}(n)=\left\{\begin{matrix} -1/4 & {\text{\rm if}} & n=1;\\
1/2 & {\text{\rm if}} & n=2;\\
0 & {\text{\rm  if}} & n>2.\end{matrix}\right.$$
Note that if $n\geq 3$, then both $J_0(n)$ and $J_{(0,-1)}(n)$ vanish in \eqref{p'2} so that $p_2'(n)=\frac{J_1(n)}{2},$ which gives \eqref{k2N}.

\subsection{The case of $3$ parts}\label{small3}
By \eqref{gen-part} and partial fraction expansion we have
\begin{eqnarray*}
\sum_{n\ge 1} p_3(n)z^n&=& \frac{z^3}{(1-z)(1-z^2)(1-z^3)}  =  \frac{z^3}{(1-z)^3(1+z)(1+z+z^2)}\\
& = &  -\frac{1}{72(1-z)}- \frac{1}{4(1-z)^2} + \frac{1}{6(1-z)^3}
 -\frac{1}{8(1+z)}\\
 &+& \frac{1}{9(1-\omega z)}+ \frac{1}{9(1-\overline{\omega} z)},
\end{eqnarray*}
where $\omega=\frac{-1+i\sqrt{3}}{2}$ is the principal $3$-root of $1$. 
Using repeatedly formula \eqref{lemma3}, after elementary simplifications we get
\begin{eqnarray*}
\sum_{n\ge 1} p_3(n)z^n&=& \sum_{n\geq 0}\left(\frac{n^2}{12}-\frac{7}{72}-\frac{(-1)^n}{8}+\frac{\omega^n+\overline{\omega}^n}{9}\right)z^n.
\end{eqnarray*}
Thus, for every $n\geq 1$, we have
\begin{equation}\label{p3}
p_3(n)=\frac{n^2}{12}-\frac{7}{72}-\frac{(-1)^n}{8}+\frac{\omega^n+\overline{\omega}^n}{9}.
\end{equation}
The above equality puts in evidence the nature of $p_3(n)$ as a sum of a polynomial and of a periodic function of period $6$ and give \eqref{part-dec}  for $k=3.$
By  \eqref{Mob-part} we then get, for every $n\geq 1$,
\begin{equation}\label{p'3}
p_3'(n)=\frac{1}{12}J_2(n)-\frac{7}{72}J_0(n) -\frac{1}{8}J_{(0,-1)}(n)+\frac{1}{9}J_{(0,\omega)}(n) +\frac{1}{9}J_{(0,\overline{\omega})}(n).
\end{equation}
Thus, we see the way in which  $p_3'(n)$ is a $\C$-linear combination of Jordan root totient functions. By Lemma \ref{comp}\,$(i)$, we have
$$\frac{-1}{8}J_{(0,-1)}(n)=\left\{\begin{matrix} 1/8 & {\text{\rm if}} & n=1;\\
-1/4 & {\text{\rm if}} & n=2;\\
0 & {\text{\rm  if}} & n\geq 3.
\end{matrix}\right.$$

Moreover, by Lemma \ref{comp}\,$(iv)$,  we have
$$\frac{1}{9}J_{(0,\omega)}(n) +\frac{1}{9}J_{(0,\overline{\omega})}(n)=\left\{\begin{matrix} -1/9 & {\text{\rm if}} & n=1;\\
0 & {\text{\rm if}} & n=2;\\
1/3 & {\text{\rm if}} & n=3;\\
0 & {\text{\rm  if}} & n\geq 4.\end{matrix}\right.$$

In particular, for $n\geq 4$, all the terms in \eqref{p'3} except the first one vanish and we get $p_3'(n)=\frac{J_2(n)}{12}$,  which confirms \eqref{k3N}.

\subsection{The case of $4$ parts}\label{small4}
By \eqref{gen-part} and partial fraction expansion we have
\begin{eqnarray*}
\sum_{n\ge 1} p_4(n)z^n&=& \frac{z^4}{(1-z)(1-z^2)(1-z^3)(1-z^4)}  \\
& = &  \frac{z^4}{(1-z)^4(1+z)^2(1+z^2)(1+z+z^2)}\\
& = & \frac{-13}{288(1-z)^2}-\frac{1}{24(1-z)^3}+\frac{1}{24(1-z)^4}+\frac{1}{32(1+z)^2} \\
& +&  \frac{1}{16(1-iz)}
+ \frac{1}{16(1+iz)} -\frac{1}{9(\omega-{\overline{\omega}})}\left(\frac{\omega}{1-\omega z}-\frac{{\overline{\omega}}}{1-{\overline{\omega}}z}\right),
\end{eqnarray*}
where $\omega=\frac{-1+i\sqrt{3}}{2}$. Now, by formula \eqref{lemma3}, we have
$$
\frac{1}{(1+z)^2}  =\sum_{n\ge 0} (-1)^{n}(n+1)
z^n,\qquad \frac{1}{(1-z)^2}=\sum_{n\ge 0} (n+1)z^n,
$$
$$
\frac{1}{(1-z)^3} =\sum_{n\ge 0} \binom{n+2}{2} z^n,\qquad \frac{1}{(1-z)^4}=\sum_{n\ge 0} \binom{n+3}{3} z^n.
$$
Hence, we get
\begin{eqnarray*}
p_4(n) & = & \frac{1}{24} \binom{n+3}{3}-\frac{1}{24}\binom{n+2}{2}-\frac{13}{288} (n+1)+\frac{(-1)^{n}(n+1)}{32}\\
& + &\frac{i^n+(-i)^n}{16}-\frac{\omega^{n+1}-{\overline{\omega}}^{n+1}}{i9\sqrt{3}}.
\end{eqnarray*}

Simplifying we obtain the expression of $p_4(n)$, for $n\geq 1:$
\begin{eqnarray}\label{p4}
p_4(n) & = & \frac{n^3}{144}+\frac{n^2}{48}-\frac{n}{32}-\frac{13}{288}+\frac{(-1)^{n}(n+1)}{32}\\
& + & \frac{i^n+(-i)^n}{16}-\frac{\omega^{n+1}-{\overline{\omega}}^{n+1}}{i9\sqrt{3}}.\nonumber
\end{eqnarray}
The above equality is \eqref{part-dec} for $k=4$ and exhibits $p_4(n)$ as the sum of a polynomial, a periodic function of period $12$ and the further term $$\frac{(-1)^{n}(n+1)}{32}$$ which is neither of polynomial type nor periodic. Writing \eqref{p4} as
\begin{eqnarray}\label{p4bis}
p_4(n) & = & \frac{n^3}{144}+\frac{n^2}{48}+\left(\frac{-1+(-1)^n}{32}\right)n-\frac{13}{288}+\frac{(-1)^{n}}{32}\\
& + & \frac{i^n+(-i)^n}{16}-\frac{\omega^{n+1}-{\overline{\omega}}^{n+1}}{i9\sqrt{3}},\nonumber
\end{eqnarray}
we instead see the explicit expression of $p_4(n)$ as a quasi--polynomial  split into its polynomial part of degree $3$ and a quasi--polynomial of degree $1$ and quasi--period $12$ as expected by Example \ref{ex}\,$(iii)$.
The expression of $p'_4(n)$, for every $n\geq 1$, follows as usual by \eqref{Mob-part} and \eqref{p4}:
\begin{eqnarray}\label{p'4}
p_4'(n)&=&\frac{J_3(n)}{144}+\frac{J_2(n)}{48}-\frac{J_1(n)}{32}-\frac{13J_0(n)}{288}+\frac{1}{32}J_{(1,-1)}(n)\\ \nonumber
&+&\frac{1}{32}J_{(0,-1)}(n)+\frac{1}{16}J_{(0,i)}(n)+\frac{1}{16}J_{(0,i^3)}(n)\\ \nonumber
&-&\frac{i\sqrt{3}+3}{54}J_{(0,\omega)}(n)+\frac{i\sqrt{3}-3}{54}J_{(0,\overline{\omega})}(n).\nonumber
\end{eqnarray}
It shows $p'_4(n)$ as a $\C$-linear combination of the Jordan root totient functions.

From the computations made in Section \ref{sum-mu}, writing $n=3^bm_1$ with $\gcd(m_1,3)=1$, we see that
$$
\frac{1}{32}J_{(1,-1)}(n)+\frac{1}{32}J_{(0,-1)}(n)=\frac{1}{32} \left\{\begin{matrix} -2 & {\text{\rm if}} & n=1;\\
5 & {\text{\rm if}} & n=2;\\
-\phi(n) & {\text{\rm if}} & n\equiv 1\pmod 2,~n>1;\\
3\phi(n)& {\text{\rm if}} & 2\| n,~n>2;\\
\phi(n) & {\text{\rm if}} & 4\mid n,\\
\end{matrix}
\right.
$$
$$
\frac{1}{16}J_{(0,i)}(n)+\frac{1}{16}J_{(0,i^3)}(n)=\frac{1}{16}\left\{\begin{matrix} 0 & {\text{\rm if}} & n=1;\\
-2 & {\text{\rm if}} & n=2;\\
4 & {\text{\rm if}} & n=4;\\
0 & {\text{\rm }} & {\text{\rm otherwise}},\end{matrix}\right.
$$
and
$$
-\frac{i\sqrt{3}+3}{54}J_{(0,\omega)}(n)+\frac{i\sqrt{3}-3}{54}J_{(0,\overline{\omega})}(n)=\frac{1}{9}
\left\{\begin{matrix} 1 & n=1;\\
-2  & n=3;\\
0 & \begin{matrix}\exists~p\equiv1 \pmod 3\\ p\mid m_1;\end{matrix}\\
0 & b\ge 2;\\
(-1)^{\Omega(n)}2^{\omega(m_1)-1} & {\text{\rm otherwise}}.\\
\end{matrix}
\right.
$$

\vspace{7mm}

\noindent {{\bf Acknowledgments}} The first author is partially supported by the research group GNSAGA of INdAM (Italy). She thanks Francesco Fumagalli for several useful conversations on the topic.
Work by the second author started during a visit at the University of Florence (Italy) in May 2015. He thanks this Institution for hospitality and INdAM  for support.
Both authors thank an anonymous referee for the advice  and many helpful comments and remarks which greatly improved content and readability of the paper.

\end{document}